\documentclass[11pt]{amsart}
\usepackage{amssymb}
\usepackage{times}
\usepackage{url}
\usepackage{hyperref}
\newtheorem{theorem}{Theorem}[section]
\newtheorem{lemma}[theorem]{Lemma}
\newtheorem{proposition}[theorem]{Proposition}
\newtheorem{corollary}[theorem]{Corollary}
\theoremstyle{remark}
\newtheorem{remark}[theorem]{Remark}

\def\gp#1{\langle \hspace*{.2mm} #1 \hspace*{.25mm} \rangle}

\newcommand{\Rad}{\mathrm{Rad}}
\newcommand{\F}{\mathbb{F}}
\newcommand{\Z}{\mathbb{Z}}
\newcommand{\Q}{\mathbb{Q}}
\newcommand{\GL}{\mathrm{GL}}

\begin{document}




\title{Algorithms for the Tits alternative and related problems}

\author{A. S. Detinko}
\address{School of Mathematics, Statistics and Applied Mathematics,
National University of Ireland, Galway, Ireland}
\email{alla.detinko@nuigalway.ie}

\author{D. L. Flannery}
\address{School of Mathematics, Statistics and Applied Mathematics,
National University of Ireland, Galway, Ireland}
\email{dane.flannery@nuigalway.ie}

\author{E. A. O'Brien}
\address{Department of Mathematics, University
of Auckland, Private Bag 92019, Auckland, New Zealand}
\email{e.obrien@auckland.ac.nz}



\begin{abstract}
We present an algorithm that decides whether a finitely generated
linear group over an infinite field is solvable-by-finite:
a computationally effective version of the Tits
alternative. We also give algorithms to decide whether the group
is nilpotent-by-finite, abelian-by-finite, or central-by-finite.
Our algorithms have been implemented in {\sc Magma} and are publicly available.
\end{abstract}

\maketitle

\section{Introduction}

The {\em Tits alternative}, established by Tits \cite{Tits},
states that a finitely generated linear group over a field either
is solvable-by-finite, or it contains a non-cyclic free subgroup.
This theorem partitions finitely generated linear groups into two
very different classes, which require separate treatment.
Consequently, one of the first questions that must be settled for
such a group is to determine the class of the Tits alternative to
which it belongs. In the class of groups with non-cyclic free
subgroups, some basic computational problems are undecidable in
general; whereas solvable-by-finite groups are more amenable to
computation (see \cite[Section 3]{survey}). For further discussion
of the Tits alternative, and its influence on other areas of group
theory, we refer to \cite{DixonTits}.

Algorithms to decide the Tits alternative over the rational field
$\Q$ were proposed in \cite{Beals1,Beals2}. Drawing on results of
\cite{Dixon85}, a different approach was considered in
\cite{Ostheimer}. Another algorithm for the Tits alternative in
$\GL(n,\Q)$, as well as practical algorithms to test solvability
and polycyclicity of rational matrix groups, appeared in
\cite{BettinaBjorn,BettinaBjorn2,Polenta}. We are not aware of
implementations of these algorithms to decide the Tits alternative
over $\Q$.

This paper gives a practical algorithm to decide whether a
finitely generated linear group over an arbitrary field is
solvable-by-finite. Additionally, we can test whether the group is
solvable. Our method uses congruence homomorphism techniques (see
\cite[Section 4]{survey}), which were applied previously to
special cases of the problems mentioned above; namely, deciding
finiteness and nilpotency
\cite{Draft,JSC4534,PositiveFiniteness,Finiteness}. We also rely
on two other recent developments. The first is a description by
Wehrfritz \cite{Wehrf} of congruence subgroups of
solvable-by-finite linear groups. The second is the development of
effective algorithms to construct presentations of matrix groups
over finite fields (see \cite{CT,OBriensurvey1}).

If the field is $\Q$, our algorithm to test virtual solvability is
a refinement and extension of that in \cite{BettinaBjorn}.
However, we consider finitely generated linear groups defined over
an arbitrary field (albeit possibly with a finite number of
exceptions in positive characteristic). We also solve the related
problems of deciding whether a group defined over a field of
characteristic zero is virtually nilpotent, virtually abelian, or
central-by-finite. The resulting algorithms are practical, and
implementations are publicly available in {\sc Magma}
\cite{Magma}.

We emphasize that this paper demonstrates that the various
problems of testing virtual properties are {\em decidable} for
finitely generated groups over a wide range of fields. Solvability
testing was previously known to be decidable for groups over
number fields \cite{Kopytov}.

Section~\ref{Background} sets up the background theory for our
congruence homomorphism techniques. In Section~\ref{SFsection} we
present an algorithm to decide virtual solvability. Section
\ref{CompletelyReducible} deals with the special case where the
group is completely reducible. In Section~\ref{NFAFCFsection} we
outline algorithms to decide whether a group in characteristic
zero is nilpotent-by-finite, abelian-by-finite, or
central-by-finite. Finally, we report on the {\sc Magma}
implementation of our algorithms.

\section{Congruence homomorphisms and computing in solvable-by-finite groups}
\label{Background}

We start by fixing some notation. Let $G = \gp{S} \leq
\mathrm{GL}(n, \F)$, where $S = \{g_1, \ldots, g_r\}$ and $\F$ is
an infinite field. Denote the integral domain generated by the
entries of the matrices in $S \cup S^{-1}$ by $R$. Recall that
$R/\rho$ is a finite field if $\rho$ is a maximal ideal of $R$
\cite[4.1, \mbox{p.} 50]{Wehrfritz}. Let $\rho$ be a (proper)
ideal of a subring $\Delta$ of $\F$; then natural projection
$\Delta \rightarrow \Delta/\rho$ extends to a group homomorphism
$\mathrm{GL}(n, \Delta) \rightarrow \mathrm{GL}(n, \Delta/\rho)$
and a ring homomorphism $\mathrm{Mat}(n, \Delta) \rightarrow$
$\mathrm{Mat}(n, \Delta/\rho)$. We denote all these homomorphisms
by $\psi_\rho$. The kernel of $\psi_\rho$ on $G$ is denoted
$G_\rho$, and is called a {\em congruence subgroup} of $G$.

\subsection{Congruence subgroups of solvable-by-finite groups}

Each solvable-by-finite linear group has a triangularizable normal
subgroup of finite index \cite[Theorem 7, \mbox{p.} 135]{Supr}; in
particular, its Zariski connected component is
unipotent-by-abelian. Proving that $G$ is solvable-by-finite is
therefore equivalent to proving that $G$ has a
unipotent-by-abelian normal subgroup of finite index. So to apply
congruence homomorphism techniques to computing in the first class
of the Tits alternative, we should first answer the following
question: if $G$ is solvable-by-finite, for which ideals $\rho
\subseteq R$ is $G_\rho$ unipotent-by-abelian? We summarize recent
results of Wehrfritz \cite[Theorems 1--3]{Wehrf} that describe
such ideals (as usual, $H'$ is the commutator subgroup $[H,H]$ of
a group $H$).
\begin{theorem}\label{wstuff}
Suppose that $G \leq \mathrm{GL}(n, \Delta)$ is
solvable-by-finite, where $\Delta$ is an integral domain.
\begin{itemize}
\item[{\rm (i)}] Let $\rho$ be an ideal of $\Delta$. If
$\mathrm{char} \, \Delta = p > n$, or $\mathrm{char} \, \Delta =
0$ and $\mathrm{char}( \Delta/\rho )=$ $p > n$, then $G'_{\rho}$
is unipotent.

\item[{\rm (ii)}] Suppose that $\Delta$ is a Dedekind domain of
characteristic zero, and $\rho$ is a maximal ideal of $\Delta$. If
$p\in \Z$ is an odd prime such that $p \in \rho \setminus
\rho^{p-1}$, then $G_\rho$ is connected; hence $G'_{\rho}$ is
unipotent.
\end{itemize}
\end{theorem}

We call $\psi_\rho: \GL(n,\Delta)\rightarrow
\GL(n,\Delta/\rho)$ a {\em W-homomorphism} if $\Delta/\rho$ is
finite and $G_\rho'$ is unipotent whenever $G\leq \GL(n,\Delta)$
is solvable-by-finite.

\subsection{Construction of W-homomorphisms}
\label{ConstructionofWhomomorphismsSection}

We may assume that $\F$ is finitely generated over its prime
subfield, and is the field of fractions of $R$. Then it suffices
to let $\F$ be one of
\begin{itemize}
\item[I.] the rationals $\Q$, \item[II.] a number field,
\item[III.] a function field $\mathbb{P}(x_1, \ldots, x_m)$, or
\item[IV.] a finite extension of $\mathbb{P}(x_1, \ldots, x_m)$,
\end{itemize}
where $\mathbb{P}$ is a number field or finite field in III--IV.
See \cite[Section 4]{survey} for more details.

In each case I--IV we explain below how to construct
W-homomorphisms on $\GL(n,R)$. Note that if $\mathbb F$ has
positive characteristic at most $n$, then in general we cannot
construct a W-homomorphism. For a subring $\Delta$ of a field,
$\frac{1}{\mu}\Delta$ denotes the localization $\{ x\mu^{-i} \mid
x \in \Delta, i \geq 0\}$ of $\Delta$ at a non-zero element $\mu$.

\subsubsection{The rational field} \label{ratfield}
(\mbox{Cf.} \cite[Lemma 9]{Dixon85}.) Let $\F = \Q$. Then $R =
\frac{1}{\mu}\Z$ for some $\mu \in \Z\setminus \{ 0\}$ determined
by the denominators of entries in the elements of $S\cup S^{-1}$.
By Theorem~\ref{wstuff} (ii), if $p\in \Z$ is an odd prime not
dividing $\mu$, then reduction mod $p$ is a W-homomorphism from
$\GL(n,R)$ onto $\GL(n,p)$. We denote this homomorphism by
$\Psi_1=\Psi_{1,p}$.

\subsubsection{Number fields} Let $\F = \Q(\alpha)$
where $\alpha$ is an algebraic integer. We may take $R =
\frac{1}{\mu}\mathbb{Z[\alpha]}$, $\mu \in \Z \setminus \{ 0\}$.
Let $f(t) = a_0 + \cdots + a_{k-1}t^{k-1} + t^k\in \Z[t]$ be the
minimal polynomial of $\alpha$. For a prime $p \in \Z$ not
dividing $\mu$, define $\psi_{2,p}:R\rightarrow
\Z_p(\bar{\alpha})$ by
\[
\psi_{2,p}: \textstyle{\sum^{k-1}_{i = 0}}b_i\alpha^i \mapsto
\textstyle{\sum^{k-1}_{i = 0}}\bar{b}_i\bar{\alpha}^i
\]
where $\bar{b}_i$ denotes the reduction of $b_i$ mod $p$, and
$\bar{\alpha}$ is a root of $\bar{f}(t) = \bar{a}_0 + \cdots +
\bar{a}_{k-1}t^{k-1} + t^k$.

\begin{lemma}\label{numberfieldswhom}
\begin{itemize}
\item[{\rm (i)}] Let $p\in \Z$ be an odd prime dividing neither
$\mu$ nor the discriminant of $f(t)$. Then $\psi_{2,p}$ is a
W-homomorphism. \item[{\rm (ii)}] Let $p\in \Z$ be a prime greater
than $n$ not dividing $\mu$. Then $\psi_{2,p}$ is a
W-homomorphism.
\end{itemize}
\end{lemma}
\begin{proof}
(i) Let $\mathcal{O}$ be the ring of integers of $\F$. Select an
irreducible factor $\bar{f}_j(t)$ of $\bar{f}(t)$, and let
$f_j(t)$ be a pre-image of $\bar{f}_j(t)$ in $\mathbb Z[t]$. The
ideal $\rho$ of $\frac{1}{\mu}\mathcal{O}$ generated by $p$ and
$f_j(\alpha)$ is maximal, and $p\not \in \rho^2$ (see
\cite[Proposition 3.8.1, Theorem 3.8.2]{Koch}). Since the kernel
of $\psi_{2,p}$ on $\mathrm{GL}(n,R)$ is contained in the kernel
of $\psi_\rho$ on $\mathrm{GL}(n,\frac{1}{\mu}\mathcal{O})$,
Theorem~\ref{wstuff} (ii) implies that $\psi_{2,p}$ is a
W-homomorphism.

(ii) This part is immediate from Theorem~\ref{wstuff} (i).
\end{proof}

For example, let $\F$ be the $c$th cyclotomic field; if $p$ is an
odd prime not dividing $\mathrm{lcm} (\mu, c)$, then $\psi_{2,p}$
is a W-homomorphism.

We denote the W-homomorphism $\psi_{2,p}$ for $p$ as in
Lemma~\ref{numberfieldswhom} by $\Psi_2= \Psi_{2,p}$.

\subsubsection{Function fields}
\label{functionfields}

Let $\F=\mathbb{P}(x_1, \ldots, x_m)$, so $R \subseteq
\frac{1}{\mu}\mathbb{P}[x_1, \ldots, x_m]$ for some
$\mathbb{P}$-polynomial $\mu = \mu(x_1, \ldots, x_m)$. Suppose
that $\alpha = (\alpha_1, \ldots, \alpha_m)$ is a non-root of
$\mu$,  where the $\alpha_i$ are in the algebraic closure
$\overline{\mathbb{P}}$ of $\mathbb{P}$. Note that if $\mathbb{P}$
is infinite then $\alpha$ can always be chosen in $\mathbb{P}^m$.
Define $\psi_{3,\alpha}$ to be the substitution homomorphism that
replaces $x_i$ by $\alpha_i$, $1\leq i\leq m$.

Let $\mathrm{char} \,  R = 0$. Set $\Psi_3=\Psi_{3,\alpha , p}=
\Psi_{i,p}\circ \psi_{3,\alpha}$, where $p>n$, $i = 1$ if
$\mathbb{P} = \Q$, and $i = 2$ if $\mathbb{P}\neq \Q$ is a number
field.

If $\mathrm{char} \,  R = p>n$ then set $\Psi_3=\Psi_{3,\alpha} =
\psi_{3,\alpha}$.

In all cases $\Psi_3$ is a W-homomorphism by Theorem~\ref{wstuff} (i).

\subsubsection{Algebraic function fields} \label{algfunfields}
Let $\F =\mathbb{L}(\beta)$ where $\mathbb{L} = \mathbb{P}(x_1,
\ldots, x_m)$, $|\F/\mathbb{L}| = e$ and $\beta$ has minimal
polynomial $f(t) = a_0 + \cdots + a_{e-1}t^{e-1}+t^e$. Then $R
\subseteq \frac{1}{\mu}\mathbb{L}_0[\beta]$ for some $\mu \in
\mathbb{L}_0 =$ $\mathbb{P}[x_1, \ldots, x_m]$. We may assume that
$f(t) \in \mathbb{L}_0[t]$.

Define $\psi_{4,\alpha}$ on $\mathrm{GL}(n,R)$ as follows. Let
$\alpha \in \overline{\mathbb{P}}^m$, $\mu(\alpha) \neq 0$; and
let $\tilde{\beta}$ be a root of $\tilde{f}(t) =$ $\tilde{a}_0 +
\cdots + \tilde{a}_{e-1}t^{e-1}+t^e$ where $\tilde{a}_i:=
\psi_{3,\alpha}(a_i)$. Each element of $R$ may be uniquely
expressed as $\sum^{e-1}_{i = 0} c_i\beta^i$ for some $c_i \in
\frac{1}{\mu}\mathbb{L}_0$. Then
\[
\psi_{4,\alpha}: \textstyle{\sum^{e-1}_{i = 0}} c_i\beta^i \mapsto
\textstyle{\sum^{e-1}_{i = 0}}\tilde{c}_i\tilde{\beta}^i
\]
where $\tilde{c}_i = \psi_{3,\alpha}(c_i)$.

Suppose that $\mathrm{char}\, R = 0$, so we can choose $\alpha \in
\mathbb{P}^m$. Set $\Psi_4= \Psi_{4,\alpha,p} =
\Psi_{i,p}\circ\psi_{4,\alpha}$ where $p>n$, $i = 1$ if
$\mathbb{P}=\Q$ and $\tilde{\beta} \in \Q$, and $i = 2$ otherwise.

If $\mathrm{char} \,  R = p > n$ then set $\Psi_{4} = \psi_{4,\alpha}$.

By Theorem~\ref{wstuff} (i), $\Psi_4$ is a W-homomorphism.

\begin{remark}\label{SWremark}
An {\em SW-homomorphism} on $\mathrm{GL}(n,R)$ is a congruence
homomorphism with finite image such that every torsion element of
its congruence subgroup is unipotent (see \cite[4.8, \mbox{p.}
56]{Wehrfritz} and \cite[Section 4]{survey}). This property of the
congruence subgroup is crucial to the algorithms of
\cite{Finiteness} for finiteness testing and structural analysis
of finite matrix groups over infinite fields. The W-homomorphisms
$\Psi_i$ are SW-homomorphisms; moreover, this remains true for
$\Psi_3$ and $\Psi_4$ without requiring that $p>n$.
\end{remark}

\section{Testing virtual solvability}
\label{SFsection}

\subsection{Preliminaries}

If $\psi_{\rho}$ is a W-homomorphism on $\GL(n,R)$, then $G$ is
solvable-by-finite if and only if $G_{\rho}'$ is unipotent. In
this subsection we develop procedures to test whether a finitely
generated subgroup of $\GL(n,R)$ is unipotent-by-abelian. Denote
the $\F$-enveloping algebra of $M\subseteq \mathrm{Mat}(n,\F)$ by
$\gp{M}_\F$, and the $\F$-linear span of $M$ by
$\mathrm{span}_\F(M)$.
\begin{lemma}\label{threepointone}
Let $H \leq  \mathrm{GL}(n, \F)$ be unipotent-by-abelian. Then $gh
- hg\in \mathrm{Rad}\gp{H}_{\F}$ for all $g, h \in H$.
\end{lemma}
\begin{proof} (\mbox{Cf.} \cite[\mbox{p.} 256]{Dixon85} and
\cite[Lemma 5]{BettinaBjorn}.) Since $H'$ is unipotent, $h_1 =
[g,h] - 1_n$ is nilpotent. For every $a \in \gp{H}_{\F}$, the
matrix $ah_1$ is nilpotent (as $H$ is triangularizable), and so
$h_1 \in\Rad \gp{H}_{\F}$. Thus $gh-hg= hgh_1 \in
\Rad\gp{H}_{\F}$.
\end{proof}

\begin{lemma}\label{3dot2}
Let $H\unlhd G$ where $H$ is unipotent-by-abelian. If $x \in
\mathrm{Rad}\gp{H}_{\F}$ then there is a non-zero $G$-module in
the nullspace of $x$.
\end{lemma}
\begin{proof}
The hypotheses on $H$ ensure that $x^g \in
\mathrm{Rad}\gp{H}_{\F}$ for all $g\in G$. Thus, the nullspace of
$\mathrm{Rad}\gp{H}_{\F}$ is a (non-zero) $G$-module in the
nullspace of $x$.
\end{proof}

In \cite[\mbox{p.} 4155]{PositiveFiniteness} we describe a simple
recursive procedure ${\tt ModuleViaNullSpace}(S, x)$ that finds,
in no more than $n$ iterations, a $G$-module $U$ in the nullspace
of $x\in \mathrm{Mat}(n,\F)$ that contains every such $G$-module.
Hence, if $x$ is as in Lemma~\ref{3dot2} then $U$ is non-zero.

We now establish a convention. For a subset $K = \{h_1, \ldots,
h_k\}$ of $\mathrm{Mat}(n, \F)$, define
\[
K^G = \{h_1^g, \ldots , h_k^g \mid g\in G\}.
\]
If $K\subseteq G$ then $\gp{K^G}$ is the normal closure of
$\gp{K}$ in $G$, which is usually denoted $\gp{K}^G$.

We next state a procedure that will be needed in several places
later.

\bigskip

${\tt BasisAlgebraClosure}(K, S)$

\medskip

Input:  finite subsets $K$ and $S = \{g_1, \ldots, g_r\}$ of
$\mathrm{GL}(n, \F)$.

Output: A basis of the $\F$-enveloping algebra of $\gp{K^G}$,
where $G = \gp{S}$.

\medskip

\begin{enumerate}

\item $\mathcal{A}:= K\cup K^{-1}$.

\item While $\exists \, g \in S \cup S^{-1}$ and $A \in
\mathcal{A}$ such that $g^{-1} A g \notin
\mathrm{span}_{\F}(\mathcal{A})$, do

$\mathcal{A} :=\mathcal{A} \cup \{ g^{-1} A g \}$.

\item `Spin up' to construct a basis $\mathcal{B}$ of the
$\F$-enveloping algebra of $\gp{\mathcal A}$.

\item Return $\mathcal{B}$.

\end{enumerate}


${\tt BasisAlgebraClosure}$ terminates in at most $n^2$
iterations. For a discussion of the well-known `spinning up'
method in step (3), see, e.g., \cite[Section 3.1]{JSC4534}. One
feature of ${\tt BasisAlgebraClosure}$ is that the basis
$\mathcal{B}$ returned consists of elements of $\gp{K^G}$.
\begin{remark}\label{obviousmods}
If $K \subseteq \mathrm{Mat}(n, \F)$ contains non-invertible
elements, then the obvious modifications should be made to ${\tt
BasisAlgebraClosure}$. That is, $\mathcal{A}$ is initialized to
$K$ in step (1); and in step (3) a basis of $\gp{\mathcal A}_\F$
is constructed (by the same spinning up as before). The output of
this modified procedure, which we name ${\tt
BasisAlgebraClosure}^*$, is a basis of $\gp{K^G}_\F$.
\end{remark}

\subsection{Testing virtual solvability}

Let $U$ be a $H$-submodule of $V := \F^n$, where $H\leq
\mathrm{GL}(n,\F)$. Extend a basis of $U$ to one of $V$, with
respect to which $H$ has block triangular form. We denote the
projection homomorphism of $H$ onto the corresponding block
diagonal group in $\mathrm{GL}(n, \F)$ by $\pi_U$. The kernel of
$\pi_U$ is a unipotent normal subgroup of $H$.

${\tt NormalGenerators}$ is a procedure that accepts $S$ and a
W-homomorphism $\Psi=\psi_\rho$ as input, and returns {\it normal
generators} for $G_\rho$, i.e., generators for a subgroup whose
normal closure in $G$ is $G_\rho$. This procedure first finds a
presentation $\mathcal P$ of $\Psi(G)$ on the generating set
$\Psi(g_1), \ldots , \Psi(g_r)$. Such presentations can be
computed using algorithms from \cite{CT,OBriensurvey1}. The
relators in $\mathcal P$ are then evaluated by replacing each
occurrence of $\Psi(g_i)$ in each relator by $g_i$, $1\leq i \leq
r$. The resulting words in the $g_i$ constitute the output of
${\tt NormalGenerators}$.

We also need the following recursive procedure.

\bigskip

${\tt ExploreBasis}(\mathcal{A},T)$

\medskip

Input: finite subsets $\mathcal{A}$, $T$ of $\mathrm{GL}(m,\F)$,
where $\mathcal{A} \subseteq \gp{T}$.

Output: ${\tt true}$ or ${\tt false}$.

\medskip

\begin{enumerate}
\item If $[A_i, A_j]= 1_m$ $\forall \, A_i, A_j \in \mathcal{A}$
then return ${\tt true}$.

\item $U_1:={\tt ModuleViaNullSpace}(T, A_iA_j - A_jA_i)$ where $[A_i, A_j] \ne 1_m$.\\
 If $U_1 = \{0 \}$ then return ${\tt false}$.

\item $\pi := \pi_{U_1}$, $U_2:=V/U_1$.

\item For $\ell = 1, 2$ do

   $\mathcal{A}_\ell := \{ \pi(A_j)_{|U_\ell} \mid A_j \in
   \mathcal{A}\}$, $T_\ell := \{ \pi(h_j)_{|U_\ell} \mid h_j \in T\}$;

   if ${\tt ExploreBasis}(\mathcal{A}_\ell,T_\ell) = {\tt false}$ then return ${\tt false}$.

\item
   Return ${\tt true}$.
\end{enumerate}

\vspace*{2.5mm}

Now we can assemble our algorithm to decide the Tits alternative.

\bigskip

${\tt IsSolvableByFinite}(S)$

\medskip

Input: $S = \{g_1, \ldots, g_r\}\subseteq \mathrm{GL}(n,R)$.

Output: ${\tt true}$ if $G = \gp{S}$ is solvable-by-finite and
${\tt false}$ otherwise.

\medskip

\begin{enumerate}

\item $K := {\tt NormalGenerators}(S,\Psi)$, $\Psi$ a
W-homomorphism on $\GL(n,R)$.

\item $\mathcal{A} := {\tt BasisAlgebraClosure}(K,S)$.

\item Return ${\tt ExploreBasis}(\mathcal{A},S)$.

\end{enumerate}

\vspace*{1.5mm}

\begin{remark}
When $\F=\Q$, ${\tt IsSolvableByFinite}$ is similar to the
algorithm of \cite[\mbox{p.} 1280]{BettinaBjorn}---but see the
first paragraph of \cite[Section 10.1]{BettinaBjorn}.
\end{remark}

${\tt IsSolvableByFinite}$ terminates in no more than $n$
iterations at step (3). A report of ${\tt false}$ is correct by
Lemmas~\ref{threepointone} and \ref{3dot2}. Note that if ${\tt
true}$ is returned at the first pass through step (1) of ${\tt
ExploreBasis}$, then $G$ is abelian-by-finite.

Algorithms to test solvability of matrix groups over finite fields
are implemented in \cite{Polenta, Magma}. We can augment ${\tt
IsSolvableByFinite}$ by checking solvability of $\Psi(G)$ during
step (1), and thus obtain a solvability testing algorithm for
finitely generated subgroups of $\GL(n,\F)$. Moreover, when
$R=\Z$, these algorithms decide whether $G$ is polycyclic or
polycyclic-by-finite (\mbox{cf.} \cite[Theorem
4.2]{Baumslagetal}).

We now point out some further additions to our basic method for
deciding virtual solvability.

First suppose that $\mathrm{char}\, \F = 0$. Sometimes we can
quickly detect that $G$ is not solvable-by-finite, by means of the
following observations. A classical theorem of Jordan states that
there is a function $f:\mathbb{N}\rightarrow \mathbb{N}$
(independent of $\mathbb{F}$) such that if $G$ is a finite
subgroup of $\mathrm {GL}(n,\F)$, then $G$ has an abelian normal
subgroup of index bounded by $f(n)$. It follows from \cite[10.11,
\mbox{p.} 142]{Wehrfritz} that if $G$ is solvable-by-finite, then
the solvable radical of $\Psi(G)$ has index bounded by $f(n)$. To
apply this criterion, we use an algorithm described in
\cite[Section 4.7.5]{HoltEicketal05} to compute the index of the
solvable radical of a matrix group over a finite field, and then
we compare this index with $f(n)$. Collins \cite{Collins2} has
found the optimal function $f$ for all $n$. In particular,
$f(n)=(n+1)!$ for $n\geq 71$.

Next, recall that if $\Psi=\psi_\rho$ is $\Psi_{3,\alpha, p}$ or
$\Psi_{4, \alpha, p}$, then $p$ must be greater than $n$ by
definition. However, with extra restrictions in place, it is
possible to test virtual solvability in characteristic $p\leq n$
too. Suppose that $\rho$ is a proper ideal of $R$ such that either
(i) $\mathrm{char}\, R=0$, $\mathrm{char} (R/\rho) >0$ and
$G_\rho$ is generated by unipotent elements; or (ii)
$\mathrm{char}\, R>0$ and $G_\rho$ is generated by diagonalizable
elements. Then $G$ is solvable-by-finite if and only if $G_\rho'$
is unipotent: this follows from the last paragraph of
\cite[Section 1]{Wehrf}, and \cite[Theorem 1 (d)]{Wehrf}. We can
determine whether (i) or (ii) holds by checking whether each
normal generator of $G_\rho$ is unipotent or diagonalizable.

\section{Completely reducible groups}
\label{CompletelyReducible}

Some of our problems coincide in an important special case.
\begin{lemma}\label{AllTheSameIfCRed}
Suppose that $G\leq \GL(n,\F)$ is completely reducible, where $\F$
is any field. Then the following are equivalent:
\begin{itemize}
\item[{\rm (i)}] $G$ is solvable-by-finite; \item[{\rm (ii)}] $G$
is nilpotent-by-finite; \item[{\rm (iii)}] $G$ is
abelian-by-finite.
\end{itemize}
\end{lemma}
\begin{proof}
Trivially (iii) $\Rightarrow$ (ii) $\Rightarrow$ (i). If $G$ is
solvable-by-finite, then a normal unipotent-by-abelian subgroup of
$G$ must be abelian, because a completely reducible unipotent
group is trivial. Thus (i) implies (iii).
\end{proof}

Motivated by Lemma~\ref{AllTheSameIfCRed}, we consider how to
decide whether a solvable-by-finite group $G$ is completely
reducible. Let $\psi_\rho$ be a W-homomorphism on $\GL(n,R)$. If
$G_\rho$ is completely reducible (hence abelian) and
$\mathrm{char} \, R$ does not divide $|G : G_\rho |$, then $G$ is
completely reducible by \cite[Theorem 1, \mbox{p.} 122]{Supr}.
Therefore, in characteristic zero, $G$ is completely reducible if
and only if the elements of ${\tt BasisAlgebraClosure}(K,S)$
commute pairwise and are all diagonalizable, where $K={\tt
NormalGenerators}(S, \psi_\rho)$. If $\mathrm{char} \, R=p>0$
divides $|G : G_\rho |$, then we cannot decide complete
reducibility of $G$; otherwise we apply the characteristic zero
criterion.

A finitely generated solvable linear group may not be finitely
presentable \cite[4.22, \mbox{p.} 66]{Wehrfritz}. However, if $G$
is both solvable-by-finite and completely reducible, then $G_\rho$
is a finitely generated abelian normal subgroup of finite index.
So we can compute presentations of $G_\rho$ and $\psi_\rho(G)$,
and combine them as explained in \cite{BettinaBjorn,CT}, to obtain
a finite presentation of $G$.

\section{Testing virtual nilpotency and related algorithms}
\label{NFAFCFsection}

We now consider the problems of deciding whether a finitely
generated linear group is nilpotent-by-finite, abelian-by-finite,
or central-by-finite. Algorithms for nilpotency testing and
computing with finitely generated nilpotent groups over arbitrary
fields are given in \cite{Large,Draft}.

Henceforth $\mathrm{char}\, \F=0$ unless stated otherwise.

\subsection{Preliminaries}

\begin{lemma}\label{fourpointfive}
Let $H\leq \mathrm{GL}(n,\F)$ be nilpotent-by-finite (resp.
abelian-by-finite), $\F$ any field. If $H$ is connected then $H$
is nilpotent (resp. abelian).
\end{lemma}
\begin{proof}
(\mbox{Cf.} \cite[Lemma 9]{Dixon85}.)  Let $N\leq H$ be nilpotent
(resp. abelian) of finite index. Then the Zariski closure of $N$
in $H$ is nilpotent (resp. abelian) and contains the connected
component of $H$; see \cite[Chapter 5]{Wehrfritz}. The lemma
follows.
\end{proof}

\begin{corollary}\label{fourpointfour}
Suppose that $R$ is a Dedekind domain of characteristic zero, and
$\rho$ is a maximal ideal of $R$ such that
$\mathrm{char}(R/\rho)=p > 2$, where $p \notin \rho^{p-1}$. Then
$G \leq \mathrm{GL}(n, R)$ is nilpotent-by-finite (resp.
abelian-by-finite) if and only if $G_{\rho}$ is nilpotent (resp.
abelian).
\end{corollary}
\begin{proof}
This follows from Theorem~\ref{wstuff} (ii) and
Lemma~\ref{fourpointfive}.
\end{proof}

Denote by $g_d, g_u\in \mathrm{GL}(n,\F)$ the diagonalizable and
unipotent parts of $g \in \mathrm{GL}(n, \F)$, i.e., $g =
g_dg_u=g_ug_d$ is the Jordan decomposition of $g$. For $X\subseteq
\mathrm{GL}(n, \F)$ we put
\[
X_d = \{x_{d}\mid x\in X\} \qquad \text{and} \qquad X_u =
\{x_u\mid x \in X\}.
\]
\begin{proposition}\label{crucial}
Let $H= \gp{K^G}$, where $K$ is a finite subset of $G$. Then $H$
is nilpotent and $H'$ is unipotent if and only if $\gp{K_d^G}$ is
abelian, $\gp{K_u^G}$ is unipotent, and $[K_d^G,K_u^G] = \{
1_n\}$.
\end{proposition}
\begin{proof}
If $\gp{K_d^G}$ is abelian, $\gp{K_u^G}$ is unipotent, and these
groups centralize each other, then the group $L$ that they
generate is unipotent-by-abelian and nilpotent. Hence the same is
true for $H \leq L$.

Now suppose that $H$ is unipotent-by-abelian and nilpotent. Then
$f_d: H \rightarrow H_d$, $f_u: H \rightarrow H_u$ defined by
\[
f_d:h\mapsto h_d, \qquad f_u:h\mapsto h_u
\]
are homomorphisms by \cite[Proposition 3, \mbox{p.} 136]{Segal}.
Thus
\[
H_d = \gp{f_d(K^G)} \qquad \text{and} \qquad H_u = \gp{f_u(K^G)}.
\]
Now $h^g= h_{d}^g h_{u}^g$ and $h_{d}^g$, $h_{u}^g$ are
diagonalizable, unipotent respectively. Uniqueness of the Jordan
decomposition implies that $h_{d}^g= (h^g)_d$ and $h_{u}^g=
(h^g)_u$,  so
\[
H_d=\gp{K_d^G} \qquad \text{and} \qquad  H_u=\gp{K_u^G}.
\]
Thus $\gp{K_u^G}$ is unipotent. Since $H$ is nilpotent, $[K_d^G ,
K_u^G] = \{1_n\}$ (see \cite[Proposition 3, \mbox{p.} 136]{Segal}
again). Finally, since $\gp{K_d^G}=H_d$ is unipotent-by-abelian
and completely reducible, it must be abelian.
\end{proof}

\subsection{Nilpotent-by-finite and abelian-by-finite groups}
\label{nfandsfsubsection}

Our algorithms for deciding whether $G$ is nilpotent-by-finite or
abelian-by-finite require that $G$ be defined over a Dedekind
domain $R$. Hence they apply, for example, when $\F$ is $\Q$, a
number field, or (a finite extension of) a univariate function
field.

\begin{lemma}\label{nilpotentvsunipotent}
Let $K\subseteq \mathrm{GL}(n,\F)$, and $\widetilde{K} := \{h -
1_n \mid h\in K\cup K^{-1}\}$. Then $H=\gp{K}$ is unipotent if and
only if $\gp{\widetilde{K}}_\F$ is nilpotent.
\end{lemma}
\begin{proof}
Observe that $\gp{\widetilde{K}}_\F=\mathrm{span}_\F ( \{h-1_n
\mid h\in H\})$. Therefore, if $H$ is unipotent then $H^x$ is
unitriangular for some $x\in \mathrm{GL}(n,\F)$, so
$\gp{\widetilde{K}}_\F$ is nilpotent. Conversely, if
$\gp{\widetilde{K}}_\F$ is nilpotent then $h - 1_n$ is nilpotent
for all $h\in H$, i.e., $H$ is unipotent.
\end{proof}

Let $K$ be a finite subset of $\mathrm{GL}(n, \F)$. The procedure
${\tt IsAbelianClosure}$ determines whether $\gp{K^G}$ is abelian
by testing whether the elements of ${\tt
BasisAlgebraClosure}(K,S)$ commute pairwise. Another auxiliary
procedure is the following (recall Remark~\ref{obviousmods}).

\bigskip

${\tt IsUnipotentClosure}(K, S)$

\medskip

Input: finite subsets $K=\{h_1 \ldots, h_k \}$ and $S$ of
$\mathrm{GL}(n, \F)$, where the $h_i$ are unipotent.

Output: ${\tt true}$ if $\gp{K^G}$ is unipotent, ${\tt false}$
otherwise, where $G = \gp{S}$.

\medskip

\begin{enumerate}

\item $\widetilde{K} := \{h_{j} - 1_n \mid 1 \leq j \leq k \}$.

\item $\mathcal{B} := {\tt BasisAlgebraClosure}^*(\widetilde{K},
S)$.

\item If $|\mathcal{B}|>n(n-1)/2$, or $B$ is not nilpotent for
some $B\in \mathcal{B}$ (i.e., $B^n\neq 0_n$), then return ${\tt
false}$.

\item If $\gp{B + 1_n : B \in \mathcal{B}}$ is unipotent then
return ${\tt true}$; else return ${\tt false}$.
\end{enumerate}

\vspace*{1.5mm}

\begin{remark}
Lemma~\ref{nilpotentvsunipotent} guarantees correctness of ${\tt
IsUnipotentClosure}$. See \cite[Section 2.1]{Large} for a
procedure to test whether a finitely generated linear group is
unipotent.
\end{remark}

Let $\Psi$ be a W-homomorphism as in
Corollary~\ref{fourpointfour}. By Proposition~\ref{crucial},
we have the following algorithm to test virtual nilpotency.

\bigskip

${\tt IsNilpotentByFinite}(S)$

\medskip

Input: a finite subset $S$ of $\mathrm{GL}(n,R)$, $R$ a Dedekind
domain of characteristic zero.

Output: ${\tt true}$ if $G = \gp{S}$ is nilpotent-by-finite, and
${\tt false}$ otherwise.

\medskip

\begin{enumerate}

\item $K := \{h_1, \ldots, h_k \}= {\tt
NormalGenerators}(S,\Psi)$.

\item $K_d := \{(h_i)_{d} \mid 1 \leq i \leq k \}$, $K_u := \{
(h_i)_{u} \mid 1 \leq i \leq k\}$.

\item \label{thirdcheck} If not ${\tt IsUnipotentClosure}(K_u,S)$
or not ${\tt IsAbelianClosure}(K_d,S)$

\noindent or $[K_d^G, K_u^G] \neq \{1_n\}$ then return ${\tt
false}$; else return ${\tt true}$.

\end{enumerate}

\vspace*{1.5mm}

\begin{remark}
In step (\ref{thirdcheck}) we use the fact that $[K_d^G,K_u^G] =
\{1_n\}$ if and only if the elements of ${\tt
BasisAlgebraClosure}(K_d,S)$ commute with the elements of ${\tt
BasisAlgebraClosure}(K_u,S)$ (these two bases are already computed
in this step).
\end{remark}

Similarly, for Dedekind domains $R$ of characteristic zero, the
algorithm ${\tt IsAbelianByFinite}(S)$ decides whether $G$ is
abelian-by-finite: it returns ${\tt IsAbelianClosure}(K, S)$,
where as usual $K$ is ${\tt NormalGenerators}(S, \Psi)$.

If either of ${\tt IsNilpotentByFinite}(S)$ or ${\tt IsAbelianByFinite}(S)$
returns ${\tt true}$, then we can decide complete reducibility of $G$:
now $G$ is completely reducible if and only if $K_u = \{ 1_n\}$.

\subsection{Central-by-finite groups}

In this subsection, instead of a W-homomorphism we may use more
generally an SW-homomorphism (see Remark~\ref{SWremark}).
\begin{lemma}\label{fivepointone}
Let $H$ be a group such that $H'$ is finite. If $A$ is a
torsion-free normal subgroup of $H$, then $A$ is central.
\end{lemma}
\begin{proof}
Since $[A,H] \leq A \cap H' = \{ 1\}$, this is clear.
\end{proof}

\begin{corollary}\label{torsion}
Let $\F$ be any field of characteristic zero, and let
$\Psi=\psi_\rho$ be an SW-homomorphism on $\GL(n,R)$. Then $G\leq
\mathrm{GL}(n,\F)$ is central-by-finite if and only if $G_\rho$ is
central.
\end{corollary}
\begin{proof}
If $G$ is central-by-finite then $G'$ is finite by a result of
Schur \cite[10.1.4, \mbox{p.} 287]{Robinson}. Since $G_\rho$ is
torsion-free, it is central by Lemma~\ref{fivepointone}. The other
direction is trivial because $|G:G_\rho|$ is finite.
\end{proof}

Corollary~\ref{torsion} underpins a simple procedure ${\tt
IsCentralByFinite}(S)$ which returns ${\tt true}$ if  $[K,S]=\{
1_n\}$, where $G_\rho =\gp{K^G}$; else it returns ${\tt false}$.
Here $\F$ is any field of characteristic zero. The same procedure
works for the fields $\F$ of positive characteristic in
Sections~\ref{functionfields}--\ref{algfunfields}, provided that
$\Psi$ is a W-homomorphism as defined there and $G_\rho$ is
completely reducible (hence torsion-free).

We could also decide whether $G$ is central-by-finite by checking
whether the `adjoint' representation that arises from the
conjugation action of $G$ on $\gp{G}_\F$ has finite image (using,
e.g., the algorithms of \cite{Finiteness}), as suggested in
\cite{Beals2}. While this approach is valid for all fields $\F$,
it may involve computing with matrices of dimension $n^2$.

\section{Implementation and performance}
\label{ExperimentalResultsSection}

We have implemented our algorithms as part of the {\sc Magma}
package {\sc Infinite} \cite{Infinite}. We use the {\sc
CompositionTree} package \cite{CT,OBriensurvey1} to study
congruence images and construct their presentations.

In practice, the single most expensive task is evaluating relators
to obtain normal generators for the kernel of a W-homomorphism.

We describe below sample outputs covering the main domains and
types of groups. The experiments were performed using {\sc Magma}
V2.17-2 on a 2GHz machine. The examples are randomly conjugated so
that generators are not sparse, and matrix entries are typically
large. All (algebraic) function fields $\F$ in these examples are
univariate, and if they have zero characteristic are over $\Q$.
Since random selection plays a role in some of the algorithms,
times have been averaged over three runs. The complete examples
are available in the {\sc Infinite} package.

\begin{enumerate}
\item $G_1 \leq \GL (7, \F)$ where $\F$ is a function field of
characteristic zero. It is conjugate to an infinite monomial
subgroup of $\GL(7,\Q)$. We decide that this $4$-generator group
is abelian-by-finite in $82$s.

\item $G_2 \leq \GL(40, \F)$ where $\F$ is an algebraic function
field of characteristic zero. It is conjugate to an infinite
completely reducible nilpotent subgroup of $\GL(40,\Q)$. We decide
that this $4$-generator group is central-by-finite in $30$s.

\item $G_3 \leq \GL(56, \F)$ where $\F$ is an algebraic function
field of characteristic zero. It is conjugate to the Kronecker
product of an infinite reducible nilpotent subgroup of $\GL(8,
\Q)$ with a primitive complex reflection group from the
Shephard-Todd list. We decide that this $7$-generator group is
nilpotent-by-finite in $219$s.

\item $G_4 \leq \GL(18, \F)$ where $\F$ is a function field over
$\mathrm{GF}(19)$. It is conjugate to the Kronecker product of a
solvable subgroup of $\GL(6, 19)$ with an infinite triangular
subgroup of $\GL(3, \F)$. We decide that this $13$-generator group
is solvable in $80$s.

\item $G_5 \leq \GL(32, \F)$  where $\F$ is the fifth cyclotomic
field. It is conjugate to the Kronecker product of an infinite
solvable subgroup of $\GL(8, \Q)$ from \cite{Polenta} with a
primitive complex reflection group from the Shephard-Todd list. We
decide that this $8$-generator group is solvable-by-finite in
$90$s.

\item $G_6 \leq \GL(12, \F)$ where $\F$ is a function field of
characteristic zero. It is conjugate to $\mathrm{SL}(12, \Z)$. We
decide that this $3$-generator group is not solvable-by-finite in
$10$s.

\item $G_7 \leq \GL(32, \F)$ where $\F$ is a number field of
degree $4$ over $\Q$. It is conjugate to the Kronecker product of
$\big\langle \tiny \left(
\begin{matrix} 1 & 1 \\ 0 & 1 \end{matrix} \right), \left(
\begin{matrix} 1 & 0 \\ 2 & 1 \end{matrix} \right) \big\rangle$ with
an infinite reducible nilpotent rational matrix group. We decide
that this $4$-generator group is not solvable-by-finite in $56$s.
\end{enumerate}

\subsection*{Acknowledgment} We are very much indebted to
Professor B.~A.~F. Wehrfritz, who kindly provided us with his new
results \cite{Wehrf} on congruence subgroups of solvable-by-finite
linear groups.

\bibliographystyle{amsplain}

\end{document}